\documentclass[11pt,leqno]{article}
\usepackage{mathrsfs,amssymb,amsmath,amsthm, color}
\usepackage[dvips]{graphicx}

\makeatletter
	
	\@addtoreset{equation}{section}
\makeatother

\def\.{{\hspace{0.2mm}}}
\def\PP{{\bf P}}
\def\EE{{\bf E}}

\def\E{{\mathcal E}}
\def\F{{\mathcal F}}

\def\H{{\mathcal H}}

\def\1{{\bf 1}}

\def\wt{\widetilde}

\def\<{{\langle}}
\def\>{{\rangle}}

\newcommand{\Rd}{\mathbb{R}^d}

\newcommand{\R}{\mathbb{R}}

\newtheorem{thm}{Theorem}[section]

\newtheorem{lem}[thm]{Lemma}
\newtheorem{prop}[thm]{Proposition}
\newtheorem{cor}[thm]{Corollary}
\newtheorem{exam}[thm]{Example}

\newtheorem{rem}[thm]{Remark}
\allowdisplaybreaks

\begin{document}

\title{On a scattering length for additive functionals and spectrum of fractional Laplacian with a non-local perturbation}

\author{Daehong Kim\footnote{The first named author is partially supported by a Grant-in-Aid for
Scientific Research (C)  No.~17K05304 from Japan Society for the Promotion of Science.}~ and Masakuni Matsuura}

\date{\empty}

\maketitle
\begin{abstract}
In this paper we study the scattering length for positive additive functionals of symmetric stable processes on $\R^d$. The additive functionals considered here are not necessarily continuous. We prove that the semi-classical limit of the scattering length equals the capacity of the support of a certain measure potential, thus extend previous results for the case of positive continuous additive functionals. We also give an equivalent criterion for the fractional Laplacian with a measure valued non-local operator as a perturbation to have purely discrete spectrum in terms of the scattering length, by considering the connection between scattering length and the bottom of the spectrum of Schr\"odinger operator in our settings.
\end{abstract}
%
%

\section{Introduction}
In \cite{Kac, KacLut}, Kac and Luttinger studied a connection between the scattering length ${\sf \Gamma}(V)$ of a positive integrable potential $V$ and Brownian motion $B_t$ on $\R^3$. They gave a probabilistic expression of ${\sf \Gamma}(V)$ as
\begin{align*}
{\sf \Gamma}(V)=\lim_{t \to \infty}\frac{1}{t}\int_{\R^3}\left(1-\EE_x\left[e^{-\int_0^tV(B_s){\rm d}s}\right]\right){\rm d}x,
\end{align*}
where $\EE_x$ denotes the expectation of $B_t$ started at $x \in \R^3$. In addition, they proved that if $V=\1_K$ for a compact set $K \subset \R^3$ satisfying the so called Kac's regularity (the Lebesgue penetration time of $K$ by $B_t$ is the same as the hitting time of $K$), then $\lim_{p\to \infty}{\sf \Gamma}(p\1_K) = {\rm Cap} (K)$. Moreover, they conjectured that for any positive integrable function $V$ with compact support satisfying the regularity as above
\begin{align}\label{Kacconjecture}
\lim_{p\to \infty}{\sf \Gamma}(pV) = {\rm Cap} ({\rm supp}[V]).
\end{align}

Taylor \cite{Taylor:1976} developed the notion of scattering length further into a tool for studying the effectiveness of a potential as a perturbation of the Laplacian $-\Delta$ on $\R^d$. More precisely, for a positive integrable function $V$ on $\R^d$, let $U_V$ be the capacitary potential of $V$ defined by $U_V(x)=\lim_{\varepsilon \to 0}(\varepsilon + V-\Delta)^{-1}V(x)$. Taylor defined the scattering length ${\sf \Gamma}(V)$ as
\begin{align*}
{\sf \Gamma}(V)=-\int_{\R^d}\Delta U_V(x){\rm d}x
\end{align*}
and proved Kac and Luttinger's formula (\cite[Proposition 1.1]{Taylor:1976}), which makes it natural that the scattering length is analogous to the capacity. Indeed, for a compact set $K \subset \R^d$ with the Kac's regularity, the capacitary potential $U_K$ of $K$ is given by
$U_{K}(x)=1-{\EE}_x[e^{-\int_0^\infty V_K(B_t){\rm d}t}]$ for $V_K=\infty$ on $K$ and $0$ off $K$. Then $-\Delta U_K$ is equal to $\gamma_K$ the equilibrium measure on $K$. Hence ${\sf \Gamma}(V_K)=-\int_{\R^d}\Delta U_K(x){\rm d}x =\int_{\R^d}\gamma_K ({\rm d}x)={\rm Cap} (K)$, where ${\rm Cap}$ denotes the Wiener capacity. Therefore, the phenomenon of \eqref{Kacconjecture} is expected naturally. 

The Kac-Luttinger's conjecture \eqref{Kacconjecture} was confirmed by Taylor \cite{Taylor:1976}. For more general symmetric Markov processes, Takahashi \cite{Takahashi} gave a new probabilistic representation of the scattering length of a continuous potential which makes the limit \eqref{Kacconjecture} quite transparent. For symmetric Markov processes again, Takeda \cite{Takeda:2010} considered the behaviour of the scattering length of a positive smooth measure potential by using the random time change argument for Dirichlet forms and gave a simple elegant proof of the analog of \eqref{Kacconjecture} without Kac's regularity. The result in \cite{Takeda:2010} was extended to a non-symmetric case by He \cite{He}. For general right Markov processes, Fitzsimmons, He and Ying \cite{FitzPY} extended Takahashi's result by using the tool of Kutznetsov measure and proved the analog of \eqref{Kacconjecture} for a positive continuous additive functional. 

Some interesting applications for spectral properties of $-\Delta +V$ have been studied with scattering lengths. Taylor \cite{Taylor:1976} gave a two-sided bound for the bottom of the spectrum of $-\Delta+V$ on a bounded region with the Neumann boundary condition via ${\sf \Gamma}(V)$. This result was extended by Siudeja \cite{Siudeja:Illinois} in the context of isotropic stable processes. Furthermore, Taylor \cite{Taylor:2006} gave the following equivalent criterion for discreteness of the spectrum of $-\Delta+V$ in terms of ${\sf \Gamma}(V)$: for given $c>0$, there exists $r_0=r_0(c) \in (0,1]$ and $R : (0,r_0] \to (0,\infty)$ such that
\begin{align}\label{equicriterion}
{\sf \Gamma} (r^2V_{r,\xi}) \ge r^2c, \quad {\rm for}~~r \in (0,r_0], ~|\xi| \ge R(r),
\end{align}  
where $V_{r,\xi}$ is the function supported on the unit cube $D_{1,0}$ in $\R^d$ defined by $V_{r,\xi}(x)=V(rx +\xi)$. In the case $V=\infty$ on $\R^d \setminus \Omega$, the condition \eqref{equicriterion} becomes ${\rm Cap} (D_{r,\xi}\setminus \Omega) \ge r^2c~ {\rm Cap} (D_{r,\xi})$ for $r \in (0,r_0]$, $|\xi| \ge R(r)$, where $D_{r,\xi}$ is the cube in $\R^d$ with side length $r$ and center $\xi$. This is known as one of the equivalent criteria for discreteness of the spectrum of $-\Delta$ on $L^2(\Omega)$ with the Dirichlet boundary condition on $\partial \Omega$ (\cite{MazyaShubin}). 
\vskip 0.2cm
Scattering lengths cited so far were considered for positive continuous additive functionals. In the present paper, we first define the scattering length of a positive additive functional of the form 
\begin{align}\label{AF}
A_t^\mu + \sum_{0<s\le t}F(X_{s-}, X_s)
\end{align}
which is not necessarily continuous, in the context of symmetric transient stable process ${\bf X}=(X_t, \PP_x)$ of index $\alpha~ (0 <\alpha < 2)$ in $\R^d$. Here $A_t^\mu$ is the positive continuous additive functional of ${\bf X}$ with a positive smooth measure $\mu$ on $\R^d$ as its Revuz measure and $F$ is a symmetric positive bounded Borel function on $\R^d\times \R^d$ vanishing on the diagonal. Let ${\bf F}^{(p)}$ be a non-local linear operator defined by 
\begin{align}\label{dfnonlocaloper}
{\bf F}^{(p)}f(x)=C_{d,\alpha}\int_{\R^d}\frac{\left(1-e^{-pF(x,y)}\right)f(y)}{|x-y|^{d+\alpha}}{\rm d}y, \quad p \ge 1 
\end{align}
for any measureable function $f$ on $\R^d$, where $C_{d,\alpha}:=\alpha 2^{\alpha -1}\pi^{-d/2}\Gamma(\frac{d+\alpha}{2})\Gamma(1-\frac{\alpha}{2})^{-1}$. Put ${\bf F}f:={\bf F}^{(1)}f$. We assume that ${\bf F}^{(p)}1 \in L^1(\R^d)$ for any $p \ge 1$.  
Let $U_{\mu+F}$ be the capacitary potential relative to the additive functional \eqref{AF} defined by
\begin{equation}
U_{\mu+F}(x):=1-\EE_x\left[e^{-A^{\mu}_{\infty} - \sum_{t>0}F(X_{t-},X_t)}\right].
\label{capacitary}
\end{equation} 
In this paper, we shall define the \emph{scattering length} ${\sf \Gamma}(\mu+F)$ relative to \eqref{AF} by 
\begin{equation*}
{\sf \Gamma}(\mu + F):= \int_{\R^d}(1-U_{\mu+F})(x)\mu({\rm d}x)+\int_{\R^d}{\bf F}(1-U_{\mu+F})(x){\rm d}x.
\end{equation*}
We will explain in Section \ref{sec:propertiesofscattering} why the expression as above is natural for the definition of the scattering length relative to \eqref{AF}. We will also give another expression for the scattering length above, which plays a crucial role throughout this paper (see Lemma \ref{anotherexpression}(1)). 

Our first result of this paper is about the semi-classical limit of the scattering length. We will investigate the behaviour of the scattering length ${\sf \Gamma}(p\mu+pF)$ when $p \to \infty$. More precisely, let $\tau_t$ be the right continuous inverse of the positive continuous additive functional $A_t^{\mu+{\bf F}1}:=A_t^{\mu}+\int_0^t{\bf F}1(X_s){\rm d}s$ defined by $\tau_t:=\inf \{s>0 \mid A_s^{\mu+{\bf F}1} >t\}$. Denote by ${\sf S}_{\mu+{\bf F}1}$ the fine support of $A_t^{\mu+{\bf F}1}$, the topological support of $\mu+{\bf F}1$ relative to the fine topology of ${\bf X}$, 
\begin{align*}
{\sf S}_{\mu+{\bf F}1}=\left\{x \in \R^d \mid \PP_x(\tau_0 =0)=1\right\}.
\end{align*}

\begin{thm}\label{behaviourofscattering}
Assume that there exists a positive function $\psi(p)$ satisfying $\psi(p) \le p$ and $\psi(p) \to \infty$ as $p \to \infty$, the non-local operator ${\bf F}^{(p)}$ induced by $F$ satisfies the following condition: for large $p\ge 1$ and a constant $C>0$
\begin{align}\label{lowerpsi}
{\bf F}^{(p)}1(x) \ge C\psi(p){\bf F}1(x) \qquad \text{for} \hspace{0.2cm} x \in \R^d.
\end{align}
Then $\lim_{p\to \infty}{\sf \Gamma}(p\mu+pF) = {\rm Cap} ({\sf S}_{\mu+{\bf F}1})$. Here ${\rm Cap}$ is the capacity relative to the Dirichlet form $(\E, \F)$ of ${\bf X}$. 
\end{thm}
Theorem \ref{behaviourofscattering} can be regarded as a generalization of the result in \cite{Takeda:2010} (in the framework of symmetric stable processes). In Section \ref{sec:propertiesofscattering}, we will prove Theorem \ref{behaviourofscattering} with the help of Lemmas \ref{anotherexpression} and \ref{largescatter} and confirm the condition \eqref{lowerpsi} with some concrete examples of $F$s.

Sections \ref{sec:EigenEst} and \ref{sec:DiscreteSpec} are devoted to the spectral theory of the Neumann fractional Laplacian with a positive potential $V$ and a measure valued non-local operator $d{\bf F}$ defined by $d{\bf F}f(x):={\bf F}f(x){\rm d}x$ as perturbations, including an equivalent criterion for discreteness of the spectrum of the Schr${\rm \ddot{o}}$dinger operator 
\begin{align*}
{\mathcal L}_{V+F}=(-\Delta)^{\alpha /2} +V+d{\bf F}.
\end{align*}
In Section \ref{sec:EigenEst}, we give a two-side bound for the bottom of the spectrum of ${\mathcal L}_{V+F}$ with the Neumann boundary condition on the unit cube $D_{1,0}$ in $\R^d$ via the scattering length ${\sf \Gamma}(V+F)$. The proofs are analogous to corresponding results in \cite{Siudeja:Illinois} with some additional modification due to $F$. 
It is well known that an operator ${\mathcal H}$ has discrete spectrum if its spectrum set $\sigma ({\mathcal H})$ consists of eigenvalues of finite multiplicity (with the only accumulated point $\infty$). We will abbreviate this with the notation $\sigma ({\mathcal H})=\sigma_{\rm d} ({\mathcal H})$. 
We also give an equivalent criterion for $\sigma({\mathcal L}_{V+F})=\sigma_{\rm d}({\mathcal L}_{V+F})$ in terms of the scattering length for $V$ and $F$, by using the results obtained in the previous section. 

Let $F_{r,\xi}$ be the function supported on $D_{1,0}\times D_{1,0}$ defined by $F_{r,\xi}(x,y):=F(rx+\xi, ry+\xi)$. The second result of this paper is the following.

\begin{thm}\label{maindiscrete}
The following conditions are equivalent.
\begin{enumerate}
\item \label{d1} For given $c>0$, there exists $r:=r(c) \in (0,1]$ and $R:=R(c) >0$ such that
\begin{align*}
{\sf \Gamma}(r^{\alpha}V_{r,\xi}+F_{r,\xi}) \ge r^\alpha c \qquad \text{for} \hspace{0.2cm}|\xi| \ge R.
\end{align*}
\item \label{d2} For given $c>0$, there exists $r_0:=r_0(c) \in (0,1]$ and $R : (0,r_0] \to (0,\infty)$ such that
\begin{align*}
{\sf \Gamma}(r^{\alpha}V_{r,\xi}+F_{r,\xi}) \ge r^{\alpha}c \qquad \text{for}\hspace{0.2cm}|\xi| \ge R(r),~ r \in (0,r_0].
\end{align*}
\item \label{d3} $\sigma({\mathcal L}_{V+F})=\sigma_{\rm d}({\mathcal L}_{V+F})$.
\end{enumerate}
\end{thm}
In Section \ref{sec:DiscreteSpec}, we give the proof of Theorem \ref{maindiscrete}. The sufficient condition for discreteness of the spectrum in terms of the scattering length (\eqref{d1} $\Longrightarrow$ \eqref{d3}) is proved by Proposition \ref{suff}. The necessary condition (\eqref{d2} $\Longrightarrow$ \eqref{d1}) follows from Proposition \ref{upper} in combination with Lemmas \ref{precpt} and \ref{asympN}. Finally we also discuss an easier-to-handle sufficient condition for $\sigma({\mathcal L}_{V+F})=\sigma_{\rm d}({\mathcal L}_{V+F})$ based on the concept of thin at infinity studied in \cite{LSW, TTT} in our settings. 

Throughout this paper, we use $c, C, c', C', c_i, C_i~ (i=1,2,\ldots)$ as positive constants which may be different at different occurrences. For notational convenience, we let $a \wedge b:=\min\{a, b\}$ for any $a, b \in \R$.

\section{Scattering lengths}\label{sec:propertiesofscattering}  
Let ${\bf X}=(X_t, \PP_x)$ be the symmetric $\alpha$-stable process in $\R^d$ with $0 <\alpha < 2$ and $d\ge 1$, that is ${\bf X}$ is a irreducible and conservative L\'evy process whose characteristic function is given by $\exp (-t|\xi|^{\alpha})~ (\xi \in \R^d)$. For simplicity, we assume that $d>\alpha$, the transience of ${\bf X}$. Note that ${\bf X}$ admits a strictly positive joint continuous transition density function $p_t(x,y)$ on $(0,\infty) \times \R^d \times \R^d$ satisfying the following two-sided bound: for some $c >0$ which depends only on $d$ and $\alpha$,
\begin{align}\label{ptbound}
c^{-1}\left(t^{-d/\alpha} \wedge \frac{t}{|x-y|^{d+\alpha}}\right) \le p_t(x,y)\le c\left(t^{-d/\alpha} \wedge \frac{t}{|x-y|^{d+\alpha}}\right), \quad x,y \in \R^d, ~t>0 
\end{align}
(cf. \cite{ChenKumagai}). Let $(-\Delta)^{\alpha /2}$ be the fractional Laplacian on $\R^d$, the generator of ${\bf X}$. The Dirichlet form $(\E, \F)$ on $L^2(\R^d)$ associated with ${\bf X}$ (or $(-\Delta)^{\alpha /2}$) is given by
\begin{align*}
&\E(f,g)=\frac{C_{d,\alpha}}{2}\int_{\R^d}\int_{\R^d}\frac{(f(x)-f(y))(g(x)-g(y))}{|x-y|^{d+\alpha}}{\rm d}x{\rm d}y \\
&\F=\left\{f \in L^2(\R^d) : \int_{\R^d}\int_{\R^d}\frac{(f(x)-f(y))^2}{|x-y|^{d+\alpha}}{\rm d}x{\rm d}y <\infty\right\},
\end{align*} 
where $C_{d,\alpha}:=\alpha 2^{\alpha -1}\pi^{-d/2}\Gamma(\frac{d+\alpha}{2})\Gamma(1-\frac{\alpha}{2})^{-1}$.  
It is known that ${\bf X}$ has a L\'evy system $(N(x,{\rm d}y), {\rm d}t)$ where $N(x,{\rm d}y)=C_{d,\alpha}|x-y|^{-(d+\alpha)}{\rm d}y$, that is, 
\begin{align*}
\EE_x\left[\sum_{s\le t} \phi  (X_{s-}, X_s)\right] =\EE_x\left[\int^t_0\int_{\R^d}\frac{C_{d,\alpha}\phi(X_s,y)}{|X_s -y|^{d+\alpha}}{\rm d}y{\rm d}s\right]
\end{align*}
for any non-negative Borel function $\phi$ on $\Rd \times \Rd$ vanishing on the diagonal and any $x \in \Rd$ (cf. \cite{CFbook}). 

Let $\mu$ be a positive smooth measure on $\R^d$ and denote by $A_t^\mu$ a positive continuous additive functional (PCAF in abbreviation) of ${\bf X}$ in the Revuz correspondence to $\mu$: for any bounded measurable function $f$ on $\R^d$ ($\mathfrak{B}_b(\R^d)$ in notation),
\begin{align*}
\int_{\R^d}f(x)\mu({\rm d}x)=\uparrow \lim_{t\to 0}\frac{1}{t}\int_{\R^d}\EE_x\left[\int_0^tf(X_s){\rm d}A_s^\mu\right]{\rm d}x.
\end{align*} 
It is known that the family of positive smooth measures and the family of equivalence classes of the set of PCAFs are in one to one correspondence under Revuz correspondence (\cite[Theorem 5.1.4]{FOT}. Let $F(x,y)$ be a symmetric positive bounded Borel function on $\R^d\times \R^d$ vanishing on the diagonal. Then $\sum_{0<s\le t}F(X_{s-},X_s)$ is an (discontinuous) additive functional of ${\bf X}$. It is natural to consider a combination of the additive functionals of the form \eqref{AF} because the pure jump process ${\bf X}$ admits many discontinuous additive functionals.

For $\beta >0$, we define the $\beta$-order resolvent kernel $r_{\beta}(x,y)=\int_0^\infty e^{-\beta t}p_t(x,y){\rm d}t$, $x,y \in \R^d$. 
Since ${\bf X}$ is transient and $\beta \mapsto r_{\beta}(x,y)$ is decreasing, one can define the $0$-order resolvent kernel $r(x,y):=\lim_{\beta \to 0}r_{\beta}(x,y)< \infty$ for $x,y \in \R^d$ with $x\neq y$. It is known that $r(x,y)$ is nothing but the Riesz kernel, $r(x,y)=C_{d,\alpha}|x-y|^{\alpha -d}$. For a non-negative Borel measure $\nu$, we write 
\begin{align*}
R\nu(x):=
C_{d,\alpha}\int_{\R^d}|x-y|^{\alpha -d}\nu({\rm d}y)=\EE_x\left[\int_0^\infty {\rm d}A_{t}^\nu\right]:=\EE_x\left[A_{\infty}^\nu\right]
\end{align*}
and $Rf(x):=R\nu(x)$ when $\nu({\rm d}x)=f(x){\rm d}x$ for any $f \in \mathfrak{B}_b(\R^d)$.

We say that a non-negative Borel measure $\nu$ on $\R^d$ (resp. a non-negative symmetric Borel function $\phi$ on $\R^d \times \R^d$ vanishing on the diagonal) is Green-bounded ($\nu \in S_{D_0}({\bf X})$ (resp. $\phi \in J_{D_0}({\bf X})$), in notations) if 
\begin{align*}
\sup_{x \in \R^d}\EE_x\left[A_{\infty}^{\nu}\right]<\infty, \quad  \left({\rm resp.}~~~\sup_{x \in \R^d}\EE_x\left[\sum_{t > 0}\phi(X_{t-},X_t)\right]<\infty\right).
\end{align*} 
Throughout this section, we assume that $F \in J_{D_0}({\bf X})$. 
\vskip 0.1cm
Let ${\bf F}$ be a non-local linear operator defined as in \eqref{dfnonlocaloper}. We assume that $\mu$ is finite, $\mu(\R^d)<\infty$ and ${\bf F}1 \in L^1(\R^d)$. Now we shall define the \emph{scattering length} ${\sf \Gamma}(\mu+F)$ relative to the additive functional \eqref{AF}. In analogy with classical potential theory, it seems to be natural to define ${\sf \Gamma}(\mu+F)$ by the total mass of $(-\Delta)^{\alpha /2}U_{\mu+F}$,
\begin{align*}
{\sf \Gamma}(\mu+F)=\int_{\R^d}(-\Delta)^{\alpha /2}U_{\mu+F}(x){\rm d}x.
\end{align*} 
We note that the capacitary potential $U_{\mu+F}$ defined in \eqref{capacitary} satisfies the following formal equation
\begin{align}
(-\Delta)^{\alpha /2}U_{\mu+F}=(1-U_{\mu+F})\mu + {\bf F}1-d{\bf F}U_{\mu+F}.
\label{DEforUF}
\end{align}
Indeed, let $\widetilde{\bf X}=(X_t, \widetilde{\PP}_x)$ be the transformed process of ${\bf X}$ by the pure jump Girsanov transform defined by 
\begin{align}\label{PureGir}
Y_t^F:=\exp \left(-\sum_{0<s\le t}F(X_{s-},X_s) + \int_0^t{\bf F}1(X_s){\rm d}s\right), \quad t \in (0,\infty).
\end{align} 
Note that the multiplicative functional \eqref{PureGir} is a uniformly integrable martingale under the assumption (cf. \cite[Theorem 3.2]{ChenUF}), because $e^{-F}-1 \ge \delta -1$ for some $\delta >0$ by the boundedness of $F$ and 
\begin{align*}
\sup_{x \in \R^d}\EE_x\left[\int_0^\infty \!\int_{\R^d}\frac{\left(e^{-F(X_{s},y)}-1\right)^2}{|X_{s}-y|^{d+\alpha}}{\rm d}y{\rm d}t\right]  &\le \sup_{x \in \R^d}\EE_x\left[\int_0^\infty \!\int_{\R^d}\frac{F(X_{s},y)}{|X_{s}-y|^{d+\alpha}}{\rm d}y{\rm d}t\right] \\
&=\sup_{x \in \R^d}\EE_x\left[\sum_{t>0}F(X_{t-}, X_t)\right] <\infty.
\end{align*}
From this fact, we see that the transformed process $\widetilde{\bf X}$ is to be a transient and conservative symmetric stable-like process on $\R^d$ in the sense of \cite{ChenKumagai}. Let $(-\widetilde{\Delta})^{\alpha /2}$ be the generator of $\widetilde{\bf X}$. Then $(-\widetilde{\Delta})^{\alpha /2}$ is formally given by
\begin{align}\label{TildeDelta}
(-\widetilde{\Delta})^{\alpha /2} = (-\Delta)^{\alpha /2} + d{\bf F} - {\bf F}1.
\end{align}
 It is known that a PCAF of ${\bf X}$ can be regarded as a PCAF of $\wt{\bf X}$ (\cite[Lemma 2.2]{Song}). Thus we see from \cite[Lemma 4.9]{KimKuwae:TAMS} and \cite[Lemma 3.2]{Takeda:2006} that
\begin{align}\label{UFRNF1}
U_{\mu+F}(x)&=1-\EE_x\left[e^{-A_{\infty}^\mu -\sum_{t>0}F(X_{t-},X_t)}\right] \nonumber \\ 
&=1-\widetilde{\EE}_x\left[e^{-A_{\infty}^\mu -\int_0^{\infty}{\bf F}1(X_t){\rm d}t}\right] \nonumber \\
&=\widetilde{\EE}_x\left[\int_0^{\infty} e^{-A_t^\mu - \int_0^t{\bf F}1(X_s){\rm d}s}\left({\rm d}A_t^\mu +{\bf F}1(X_t){\rm d}t\right)\right].  
\end{align}
The equation \eqref{UFRNF1} implies that $U_{\mu+F}$ satisfies the following formal equation
\begin{align}\label{eqmuF1}
\left(\mu + {\bf F}1+(-\widetilde{\Delta})^{\alpha /2}\right)U_{\mu+F} = \mu + {\bf F}1.
\end{align}
Hence we have \eqref{DEforUF} by applying \eqref{TildeDelta} to \eqref{eqmuF1}. We note that the relation \eqref{DEforUF} is rigorously established for $\mu$ and $F$ whenever $\mu$ and ${\bf F}1$ are belonging to $L^2(\R^d)$. 
\vskip 0.2cm
Now, let us define the scattering length ${\sf \Gamma}(\mu+F)$ relative to \eqref{AF} by the total mass of  $(-\Delta)^{\alpha /2}U_{\mu+F}$, that is, 
\begin{equation}
{\sf \Gamma}(\mu + F):= \int_{\R^d}(1-U_{\mu+F})(x)\mu({\rm d}x)+\int_{\R^d}{\bf F}(1-U_{\mu+F})(x){\rm d}x.
\label{scattering length}
\end{equation}

The following expression and monotonicity of the scattering length \eqref{scattering length} play a crucial role throughout this paper. 

\begin{lem}\label{anotherexpression}
{\rm (1)}~~The scattering length \eqref{scattering length} also can be rewritten as
\begin{align}
{\sf \Gamma}(\mu+F)=\int_{\R^d}\EE_x\left[e^{-A_{\infty}^{\mu} - \sum_{t>0}F(X_{t-},X_t)}\right]\left(\mu({\rm d}x) +{\bf F}1(x){\rm d}x\right). \label{express} 
\end{align}
{\rm (2)}~~Let $\nu$ be a positive finite smooth measure on $\R^d$ and $G$ be a symmetric positive bounded Borel function on $\R^d \times \R^d$ vanishing on the diagonal set satisfying ${\bf G}1 \in L^1(\R^d)$, where ${\bf G}$ is the non-local operator defined as in \eqref{dfnonlocaloper} for $G$. If $\mu \le \nu$ and $F \le G$, then ${\sf \Gamma}(\mu+F) \le {\sf \Gamma}(\nu+G)$. 
\end{lem}
\vskip 0.2cm
\begin{proof}~~(1): The expression \eqref{express} is a consequence of the symmetry of $F$. Indeed,  
\begin{align*}
\int_{\R^d}{\bf F}(1-U_{\mu+F})(x){\rm d}x&=C_{d,\alpha}\int_{\R^d}\int_{\R^d}\frac{\left(1-e^{-F(x,y)}\right)\EE_y[e^{-A_{\infty}^{\mu} - \sum_{t>0}F(X_{t-},X_t)}]}{|x-y|^{d+\alpha}}{\rm d}y{\rm d}x \\
&=\int_{\R^d}\EE_x\left[e^{-A_{\infty}^{\mu} - \sum_{t>0}F(X_{t-},X_t)}\right]{\bf F}1(x){\rm d}x.
\end{align*}
(2): The proof is a mimic of the proof of \cite[Proposition 3.2]{Siudeja:Illinois} in our settings. It is clear that $U_{\mu+F} \le U_{\nu+G}$ under the assumption. By virtue of \cite[Lemma 4.9]{KimKuwae:TAMS} and \cite[Lemma 3.2]{Takeda:2006},  
\begin{align}\label{UtildeRzeta}
U_{\mu+F}(x)&=1-\EE_x\left[e^{-A_{\infty}^\mu - \sum_{t>0}F(X_{t-},X_t)}\right] =1-\widetilde{\EE}_x\left[e^{-A_{\infty}^\mu -\int_0^{\infty}{\bf F}1(X_t){\rm d}t}\right] \nonumber \\
&=\widetilde{\EE}_x\left[\int_0^{\infty} \widetilde{\EE}_{X_t}\left[e^{-A_{\infty}^\mu -\int_0^{\infty}{\bf F}1(X_s){\rm d}s}\right]\left({\rm d}A_t^\mu +{\bf F}1(X_t){\rm d}t\right)\right] \nonumber \\
&=\widetilde{\EE}_x\left[\int_0^{\infty} \EE_{X_t}\left[e^{-A_{\infty}^\mu - \sum_{s>0}F(X_{s-},X_s)}\right]({\rm d}A_t^\mu +{\bf F}1(X_t){\rm d}t)\right] \nonumber \\
&:= \widetilde{R}\left(\EE_{\cdot}\left[e^{-A_{\infty}^\mu - \sum_{t>0}F(X_{t-},X_t)}\right](\mu+ {\bf F}1)\right)(x),
\end{align}
where $\widetilde{R}$ denotes the 0-order resolvent operator of $\widetilde{\bf X}$ with the resolvent kernel $\wt{r}(x,y)$. 
Let $K \subset \R^d$ be a Kac's regular set in the sense that the Lebesgue penetration time of $K$ by the transformed process $\widetilde{\bf X}$ is the same as the hitting time of $K$. In this case, the capacity potential $\widetilde{U}_K$ of $K$ is given by 
\begin{align*}
\widetilde{U}_K(x)=1-\widetilde{\EE}_x\left[e^{-\int_0^{\infty} V_K(X_t){\rm d}t}\right]
\end{align*} 
for $V_K=\infty$ on $K$ and 0 off $K$. Moreover, analogously to \eqref{UtildeRzeta}, we see that $\widetilde{U}_K(x)=\widetilde{R}\gamma_K(x)$ for the equilibrium measure $\gamma_K$ on $K$ (cf. \cite{Stroock}).

First, we suppose that the topological supports ${\rm supp} [\mu+{\bf F}1]$ and ${\rm supp} [\nu+{\bf G}1]$ are bounded such that ${\rm supp} [\mu+{\bf F}1] \cup {\rm supp} [\nu+{\bf G}1] \subset K$. In view of \eqref{UtildeRzeta} and the fact that $\widetilde{U}_K=\widetilde{R}\gamma_K=1$ on ${\rm supp} [\mu+{\bf F}1] \cup {\rm supp} [\nu+{\bf G}1]$, we then have
\begin{align}\label{mono}
{\sf \Gamma}(\mu+F)
&=\int_{{\rm supp} [\mu+{\bf F}1]}\widetilde{R}\gamma_K(x)\EE_x\left[e^{-A_{\infty}^{\mu} - \sum_{t>0}F(X_{t-},X_t)}\right](\mu({\rm d}x) +{\bf F}1(x){\rm d}x) \nonumber \\
&=\int_{{\rm supp} [\mu+{\bf F}1]}\int_{\R^d}\widetilde{r}(x,y)\EE_x\left[e^{-A_{\infty}^{\mu} - \sum_{t>0}F(X_{t-},X_t)}\right](\mu({\rm d}x) +{\bf F}1(x){\rm d}x)\gamma_K({\rm d}y) \nonumber \\
&= \int_{{\rm supp} [\mu+{\bf F}1]}\widetilde{R}\left(\EE_{\cdot}\left[e^{-A_{\infty}^\mu - \sum_{t>0}F(X_{t-},X_t)}\right](\mu+ {\bf F}1)\right)(y)\gamma_K({\rm d}y) \nonumber \\
&= \int_{{\rm supp} [\mu+{\bf F}1]}U_{\mu+F}(y)\gamma_K({\rm d}y) \le \int_{{\rm supp} [\nu+{\bf G}1]}U_{\nu+G}(y)\gamma_K({\rm d}y)= {\sf \Gamma}(\nu+G).
\end{align}

Now we prove the monotonicity \eqref{mono} without assumptions on the supports. Let $\nu_n$ (resp. $G_n$) be a non-decreasing sequence of finite positive smooth measures on $\R^d$ (resp. a non-decreasing sequence of symmetric positive bounded Borel functions on $\R^d \times \R^d$ vanishing on the diagonal satisfying ${\bf G}_n1 \in L^1(\R^d)$) such that ${\rm supp}[\nu_n+{\bf G}_n1]$ is bounded, ${\rm supp}[\nu_n+{\bf G}_n1] \subset K$ for any $n\ge 1$ and $\nu_n + G_n \nearrow \nu +G$ as $n \to \infty$. Set $\mu_n:=\nu_n \wedge \mu$ and $F_n:=G_n \wedge F$. Then we have
\begin{align*}
&\int_{\R^d}\EE_x\left[e^{-A_{\infty}^{\mu_n} - \sum_{t>0}F_n(X_{t-},X_t)}\right](\mu_n({\rm d}x) +{\bf F}_n1(x){\rm d}x)  \\
&\quad ={\sf \Gamma}(\mu_n+F_n) \le {\sf \Gamma}(\nu_n+G_n)=\int_{\R^d}\EE_x\left[e^{-A_{\infty}^{\nu_n} - \sum_{t>0}G_n(X_{t-},X_t)}\right](\nu_n({\rm d}x) +{\bf G}_n1(x){\rm d}x).
\end{align*} 
Both integrands are bounded above by $\int_{\R^d}(\nu({\rm d}x) + {\bf G}1(x){\rm d}x) < \infty$, hence we have the assertion by the dominated convergence theorem.
\end{proof}

\begin{rem}\label{advansecond}
The scattering length \eqref{scattering length} is also expressed as 
\begin{align*}
{\sf \Gamma}(\mu+F)=\lim_{t \to \infty}\frac{1}{t}\int_{\R^d}\left(1-\EE_x\left[e^{-A_t^\mu - \sum_{0<s \le t}F(X_{s-},X_s)}\right]\right){\rm d}x.
\end{align*}
Indeed, it follows from \cite[Lemma 4.9]{KimKuwae:TAMS} and \eqref{express} that 
\begin{align*}
{\sf \Gamma}(\mu+F)&=\int_{\R^d}\widetilde{\EE}_x\left[e^{-A_{\infty}^{\mu}-\int_0^{\infty}{\bf F}1(X_s){\rm d}s}\right]\left(\mu({\rm d}x) +{\bf F}1(x){\rm d}x\right) \\
&=\lim_{t \to \infty}\frac{1}{t}\int_{\R^d}\left(1-\widetilde{\EE}_x\left[e^{-A_t^{\mu}-\int_0^t{\bf F}1(X_s){\rm d}s}\right]\right){\rm d}x \\
&=\lim_{t \to \infty}\frac{1}{t}\int_{\R^d}\left(1-\EE_x\left[e^{-A_t^\mu - \sum_{0<s \le t}F(X_{s-},X_s)}\right]\right){\rm d}x.
\end{align*}
In the second equality above, we use the result due to \cite[(2.2)]{Takeda:2010} (also \cite[Theorem 2.2]{FitzPY}).
\end{rem}

Now we are going to study the behaviour of the scattering length ${\sf \Gamma}(p\mu + pF)$ when $p \to \infty$. As we mentioned in the previous section, this problem was decisively solved in the case $F\equiv 0$ by Takeda \cite{Takeda:2010}, through the random time change argument for Dirichlet forms: let $\nu$ be a positive finite smooth measure on $\R^d$ and ${\sf S}_{\nu}$ the fine support of $A_t^\nu$. Then 
\begin{align}\label{Takedalocal}
\lim_{p \to \infty}{\sf \Gamma}(p\nu)={\rm Cap} ({\sf S}_{\nu}).
\end{align}
However, we cannot apply time change method to our problem directly because our scattering length contains a discontinuous additive functional. 

Let $\tau_t$ be the right continuous inverse of the PCAF $A_t^\mu + \int_0^t{\bf F}1(X_s){\rm d}s$, that is, $\tau_t:=\inf\{s>0 \mid A_s^\mu +\int_0^s{\bf F}1(X_u){\rm d}u >t\}$. Let denote by ${\sf S}_{\mu+{\bf F}1}$ the fine support of $A_t^\mu +\int_0^t{\bf F}1(X_s){\rm d}s$,
\begin{align*}
{\sf S}_{\mu+{\bf F}1}=\left\{x \in \R^d ~\Big|~ \PP_x(\tau_0 =0)=1\right\},
\end{align*}
and by ${\rm Cap}$ the capacity relative to the Dirichlet form $(\E, \F)$ (see \cite{CFbook}, \cite{FOT}). For $p \ge 1$, set
\begin{align*}
{\bf F}^{(p)}1(x):=C_{d,\alpha}\int_{\R^d}\frac{\left(1-e^{-pF(x,y)}\right)}{|x-y|^{d+\alpha}}{\rm d}y. 
\end{align*}
Clearly, ${\bf F}1(x) (={\bf F}^{(1)}1(x)) \le {\bf F}^{(p)}1(x)$ for $p \ge 1$. 
For notational convenience, we let $A_t^{{\bf F}1}:=\int_0^t{\bf F}1(X_s){\rm d}s$.

\begin{lem}\label{largescatter}
For any $\varepsilon >0$
\begin{align*}
\lim_{p\to \infty}{\sf \Gamma}\left(pF+p^{1+\varepsilon}\mu+p^{1+\varepsilon}{\bf F}1\right)={\rm Cap} ({\sf S}_{\mu+{\bf F}1}).
\end{align*}
In particular, $\limsup_{p\to \infty}{\sf \Gamma}(p\mu+pF) \le {\rm Cap} ({\sf S}_{\mu+{\bf F}1})$.  
\end{lem}

\begin{proof}
The last assertion easily follows from the first one with the monotonicity of the scattering length. Put $k=1/(1+\varepsilon)$. From the expression \eqref{express}, one can easily see that 
\begin{align*}
&{\sf \Gamma}\left(p^kF +p\mu+p{\bf F}1\right)  \\
&\qquad =\int_{\R^d}\EE_x\left[e^{-p^k\sum_{t>0}F(X_{t-},X_t)-pA_{\infty}^{\mu} - pA_{\infty}^{{\bf F}1}}\right]\left(p\mu({\rm d}x) +{\bf F}^{(p^{k})}1(x){\rm d}x+p{\bf F}1(x){\rm d}x\right).
\end{align*}
Since ${\bf F}^{(q)}1 \le q{\bf F}1$ for any $q\ge 1$, 
\begin{align*}
&{\sf \Gamma}\left(p^kF +p\mu+p{\bf F}1\right) \\
&\qquad \le \int_{\R^d}\EE_x\left[e^{-pA_{\infty}^{\mu} - pA_{\infty}^{{\bf F}1}}\right]\left(\left(1+p^{k-1}\right)p\mu({\rm d}x)+\left(p^k{\bf F}1+p{\bf F}1\right)(x){\rm d}x\right) \\
&\qquad = \left(1+p^{k-1}\right){\sf \Gamma}\left(p\mu+p{\bf F}1\right). \nonumber
\end{align*}
Therefore we have by the monotonicity of the scattering length that
\begin{align*}
{\sf \Gamma}\left(p\mu+p{\bf F}1\right)\le {\sf \Gamma}\left(p^kF +p\mu+p{\bf F}1\right) \le \left(1+p^{k-1}\right){\sf \Gamma}\left(p\mu+p{\bf F}1\right).
\end{align*}
By applying \eqref{Takedalocal}, the scattering lengths of both sides of the above converge to ${\rm Cap} ({\sf S}_{\mu+{\bf F}})$ as $p \to \infty$, which implies the first assertion. 
\end{proof}

Now we prove the first main result of this paper.
\vskip 0.2cm
\emph{Proof of Theorem \ref{behaviourofscattering}}: 
Let $\psi(p)$ be the function which appeared in the condition \eqref{lowerpsi}. By the monotonicity of the scattering length, we have for some $C>0$
\begin{align*}
{\sf \Gamma}\left(\frac{\psi(p)}{n}\mu+\frac{C\psi(p)}{n}{\bf F}1\right) \le {\sf \Gamma}\left(pF +\frac{\psi(p)}{n}\mu+\frac{C\psi(p)}{n}{\bf F}1\right) \le {\sf \Gamma}\left(pF+p^{1+\varepsilon}\mu+p^{1+\varepsilon}{\bf F}1\right)
\end{align*}
for any $n \ge 1$ and $\varepsilon >0$. Then, by Lemma \ref{largescatter} and applying \eqref{Takedalocal} again, one can get that 
\begin{align*}
\lim_{p\to \infty}{\sf \Gamma}\left(pF +\frac{\psi(p)}{n}\mu+\frac{C\psi(p)}{n}{\bf F}1\right) 
 ={\rm Cap} ({\sf S}_{\mu+{\bf F}1}). 
\end{align*}
From this and the condition \eqref{lowerpsi}, we see that
\begin{align*}
&\liminf_{p \to \infty}{\sf \Gamma}\left(p\mu+pF\right) \ge \liminf_{p \to \infty}{\sf \Gamma}\left(\frac{\psi(p)}{n}\mu+pF\right) \\
&\quad =\liminf_{p\to \infty}\int_{\R^d}\EE_x\left[e^{-p\sum_{t>0}F(X_{t-},X_t)-\frac{\psi(p)}{n}A_{\infty}^\mu}\right]\left(\frac{\psi(p)}{n}\mu({\rm d}x)+{\bf F}^{(p)}1(x){\rm d}x\right) \\
&\quad =\frac{n}{n+1}\liminf_{p\to \infty}\int_{\R^d}\!\EE_x\left[e^{-p\sum_{t>0}F(X_{t-},X_t)-\frac{\psi(p)}{n}A_{\infty}^\mu}\right] \\& \quad \quad \quad \quad \quad \quad \quad \quad \quad \quad \quad \quad \quad \quad \quad \quad \cdot\left(\frac{n+1}{n}\frac{\psi(p)}{n}\mu({\rm d}x)+{\bf F}^{(p)}1(x){\rm d}x+\frac{1}{n}{\bf F}^{(p)}1(x){\rm d}x\right) \\
&\quad \ge \frac{n}{n+1} \liminf_{p\to \infty}\int_{\R^d}\EE_x\left[e^{-p\sum_{t>0}F(X_{t-},X_t)-\frac{\psi(p)}{n}A_{\infty}^\mu-\frac{C\psi(p)}{n}A_{\infty}^{{\bf F}1}}\right] \\
& \quad \quad \quad \quad \quad \quad \quad \quad \quad \quad \quad ~\quad \quad \quad \quad \quad \quad \quad \cdot \left({\bf F}^{(p)}1(x){\rm d}x+\frac{\psi(p)}{n}\mu({\rm d}x)+\frac{C\psi(p)}{n}{\bf F}1(x){\rm d}x\right) \\
&\quad = \frac{n}{n+1}\lim_{p\to \infty}{\sf \Gamma}\left(pF +\frac{\psi(p)}{n}\mu+\frac{C\psi(p)}{n}{\bf F}1\right)=\frac{n}{n+1}{\rm Cap} ({\sf S}_{\mu+{\bf F}1}).
\end{align*}
Letting $n \to \infty$, we have
\begin{align}\label{low}
{\rm Cap} ({\sf S}_{\mu+{\bf F}1}) \le \liminf_{p \to \infty}{\sf \Gamma}(p\mu+pF).
\end{align}
The proof will be finished by the last assertion of Lemma \ref{largescatter} and \eqref{low}. \qed
\vskip 0.3cm
Now, we consider some concrete examples of $F$s satisfying the condition \eqref{lowerpsi}. Denote by $B(a,b)$ the open ball in $\R^d$ with center $a$ and radius $b$.
 
\begin{exam}\label{examplprop}
Let $F$ be the function on $\R^d \times \R^d$ such that for $R, R'>0$
\begin{align*}
F(x,y)&=\frac{1}{2}\phi(|x-y|)\left({\bf 1}_{B(x,R')}(y){\bf 1}_{B(0,R)}(x)+{\bf 1}_{B(y,R')}(x){\bf 1}_{B(0,R)}(y)\right. \\
&\qquad \qquad \left.+ {\bf 1}_{B(y,R')}(x){\bf 1}_{B(x,R')}(y){\bf 1}_{B(0,R+R')\setminus B(0,R)}(x){\bf 1}_{B(0,R+R')\setminus B(0,R)}(y)\right), 
\end{align*}
where $\phi$ is a non-negative strictly increasing smooth function on $[0,\infty)$ satisfying $\phi(0)=0$, $\phi(t)=o(t^{\alpha})$ $(t \to 0)$ and 
\begin{align}
\phi^{-1}(p^{-1}t)^\alpha \le \frac{1}{\psi(p)}\phi^{-1}(t)^\alpha, \quad t\ge 0, ~{\rm large}~ p \ge 1
\label{phiinverse}
\end{align}
for a positive function $\psi(p)$ such that $\psi(p) \le p$ and $\psi(p) \to \infty$ as $p\to \infty$. 
Then the condition \eqref{lowerpsi} holds for this $F$. First, let us take $x \in B(0,R)$. In this case, $F$ is given by 
\begin{align*}
F(x,y)=
\left\{
\begin{array}{lll}
\phi(|x-y|) & y \in B(x,R') \cap B(0,R) \\
\frac{1}{2}\phi(|x-y|) & y \in B(x,R') \cap B(0,R)^c \\
0 & {\rm otherwise}.
\end{array}
\right.
\end{align*}
Then we have 
\begin{align}\label{F1upper}
{\bf F}1(x)&=C_{d,\alpha}\int_{\R^d}\frac{1-e^{-F(x,y)}}{|x-y|^{d+\alpha}}{\rm d}y=C_{d,\alpha}\int_{B(x,R')}\frac{1-e^{-F(x,y)}}{|x-y|^{d+\alpha}}{\rm d}y \nonumber \\
&=C_{d,\alpha}\left\{\int_{B(x,R')\cap B(0,R)}\frac{1-e^{-\phi(|x-y|)}}{|x-y|^{d+\alpha}}{\rm d}y+\int_{B(x,R')\cap B(0,R)^c}\frac{1-e^{-\frac{1}{2}\phi(|x-y|)}}{|x-y|^{d+\alpha}}{\rm d}y\right\} \nonumber \\
&\le C_{d,\alpha}\int_{B(x,R')}\frac{1-e^{-\phi(|x-y|)}}{|x-y|^{d+\alpha}}{\rm d}y.
\end{align}
By using integration by parts, the right-hand side of \eqref{F1upper} is equal to
\begin{align}\label{cortran}
C_{d,\alpha}'\int_0^{R'}\frac{1-e^{-\phi(r)}}{r^{1+\alpha}}{\rm d}r &=C_{d,\alpha}'\left\{\frac{e^{-\phi(R')}-1}{\alpha (R')^\alpha}+\frac{1}{\alpha}\int_0^{R'}r^{-\alpha}\phi'(r)e^{-\phi(r)}{\rm d}r\right\} \nonumber \\
&=C_{d,\alpha}'\left\{\frac{e^{-\phi(R')}-1}{\alpha (R')^\alpha}+\frac{1}{\alpha}\int_0^{\phi(R')}\frac{1}{\phi^{-1}(t)^\alpha}e^{-t}{\rm d}t\right\},
\end{align} 
where $C_{d,\alpha}'$ is a positive constant depending on $d$ and $\alpha$. On the other hand, by a similar calculation above with the inequality $1-e^{-a-b} \le (1-e^{-a})+(1-e^{-b})$ for any $a,b \ge 0$ and the condition \eqref{phiinverse}, we see 
\begin{align}\label{F1lo}
{\bf F}^{(p)}1(x)&=C_{d,\alpha}\left\{\int_{B(x,R')\cap B(0,R)}\frac{1-e^{-p\phi(|x-y|)}}{|x-y|^{d+\alpha}}{\rm d}y+\int_{B(x,R')\cap B(0,R)^c}\frac{1-e^{-\frac{p}{2}\phi(|x-y|)}}{|x-y|^{d+\alpha}}{\rm d}y\right\} \nonumber \\
&\ge C_{d,\alpha}\int_{B(x,R')}\frac{1-e^{-\frac{p}{2}\phi(|x-y|)}}{|x-y|^{d+\alpha}}{\rm d}y \ge \frac{C_{d,\alpha}}{2}\int_{B(x,R')}\frac{1-e^{-p\phi(|x-y|)}}{|x-y|^{d+\alpha}}{\rm d}y \nonumber \\
&=\frac{C_{d,\alpha}'}{2}\left\{\frac{e^{-p\phi(R')}-1}{\alpha (R')^\alpha}+\frac{1}{\alpha}\int_0^{p\phi(R')}\frac{1}{\phi^{-1}(p^{-1}t)^\alpha}e^{-t}{\rm d}t\right\} \nonumber \\
&\ge \frac{\psi(p)}{2}C_{d,\alpha}'\left\{\frac{e^{-\phi(R')}-1}{\alpha (R')^\alpha}+\frac{1}{\alpha}\int_0^{\phi(R')}\frac{1}{\phi^{-1}(t)^\alpha}e^{-t}{\rm d}t\right\} 
\end{align}
for large $p \ge 1$. Hence we can confirm by \eqref{F1upper}, \eqref{cortran} and \eqref{F1lo} that
\begin{align}\label{conclude1}
{\bf F}^{(p)}1(x) \ge \frac{1}{2}\psi(p) {\bf F}1(x), \quad x \in B(0,R), ~{\rm large}~ p\ge 1.
\end{align}

Next, take $x \in B(0,R+R')\setminus B(0,R)$. In this case, $F$ is given by 
\begin{align*}
F(x,y)=\left\{
\begin{array}{lll}
\frac{1}{2}\phi(|x-y|) & y \in B(x,R') \cap B(0,R) \\
\frac{1}{2}\phi(|x-y|) & y \in B(x,R') \cap B(0,R)^c \\
0 & {\rm otherwise}.
\end{array}
\right.
\end{align*} 
Then, by the same calculations as \eqref{F1upper}, \eqref{cortran} and \eqref{F1lo}  
\begin{align*}
{\bf F}1(x) \le C_{d,\alpha}\int_{B(x,R')}\frac{1-e^{-\phi(|x-y|)}}{|x-y|^{d+\alpha}}{\rm d}y =C_{d,\alpha}'\left\{\frac{e^{-\phi(R')}-1}{\alpha (R')^\alpha}+\frac{1}{\alpha}\int_0^{\phi(R')}\frac{1}{\phi^{-1}(t)^\alpha}e^{-t}{\rm d}t\right\}
\end{align*}
and
\begin{align*}
{\bf F}^{(p)}1(x)  
&\ge \frac{C_{d,\alpha}}{2}\int_{B(x,R')}\frac{1-e^{-p\phi(|x-y|)}}{|x-y|^{d+\alpha}}{\rm d}y \\
&\ge \frac{\psi(p)}{2}C_{d,\alpha}'\left\{\frac{e^{-\phi(R')}-1}{\alpha (R')^\alpha}+\frac{1}{\alpha}\int_0^{\phi(R')}\frac{1}{\phi^{-1}(t)^\alpha}e^{-t}{\rm d}t\right\} = \frac{\psi(p)}{2}{\bf F}1(x)
\end{align*}
for large $p \ge 1$. Therefore, we can confirm \eqref{conclude1} for $x \in B(0,R+R')\setminus B(0,R)$. For $x \in B(0,R+R')^c$,  \eqref{conclude1} is trivial because ${\bf F}^{(p)}1(x)=0$ for any $p \ge 1$. Hence we have \eqref{conclude1} for any $x \in \R^d$. Moreover, 
for $x \in B(0,R+R')$
\begin{align*}
{\bf F}^{(p)}1(x) &= C_{d,\alpha}\left(\int_{B(0,R+R')\cap B(x,1)}\frac{1-e^{-pF(x,y)}}{|x-y|^{d+\alpha}}{\rm d}y + \int_{B(0,R+R')\cap B(x,1)^c}\frac{1-e^{-pF(x,y)}}{|x-y|^{d+\alpha}}{\rm d}y\right) \\
&\le C_{d,\alpha}\int_{B(x,1)}\frac{1-e^{-\frac{3p}{2}\phi(|x-y|)}}{|x-y|^{d+\alpha}}{\rm d}y+C_{d,\alpha}\int_{\overline{B(0,R+R')}\cap B(x,1)^c}\left(1-e^{-\frac{3p}{2}\|\phi\|_{\infty}}\right){\rm d}y \\
&\le C_{d,\alpha}'\int_0^1\frac{1-e^{-\frac{3p}{2}\phi(r)}}{r^{\alpha +1}}{\rm d}r + C_{d,\alpha}\left|\overline{B(0,R+R')}\right|.
\end{align*} 
Since $\phi(t)=o(t^{\alpha})$ $(t \to 0)$, it follows that ${\bf F}^{(p)}1$ is bounded on $B(0,R+R')$ and is zero on $B(0,R+R')^c$ for any $p\ge 1$. This shows that ${\bf F}^{(p)}1 \in L^1(\R^d)$ for any $p \ge 1$. By a similar way as in the proof of \cite[Proposition 7.10(3)]{CKK}, one can also prove that ${\bf F}^{(p)}1 \in L^{\ell}(\R^d) (\ell \ge 1)$ for any $p \ge 1$. We omit the details.

There are many functions satisfying the condition \eqref{phiinverse}. For instance, they can be given by $\phi(t)=t^{\beta}$, $\phi(t)=t^\beta / (1+t)^\beta$, $\phi(t):=\phi^{(1)}(t)=\log (1+t^\beta)$ and its iterated function $\phi^{(n)}(t)=\phi(\phi^{(n-1)}(t))$ ($n\ge 2$) for $\beta > \alpha$. For these functions, it holds that $F \in J_{D_0}({\bf X})$ (cf. \cite[Example 2.1]{ChenAnalyticF}) and we can take the function $\psi(p)$ which appeared in \eqref{phiinverse} as 
\begin{align*}
\psi(p)=p^{\alpha /\beta}.
\end{align*}
Hence, the scattering length ${\sf \Gamma}(p\mu + pF)$ induced by the functions $\phi$ above converges to ${\rm Cap}({\sf S}_{\mu+{\bf F}1})$ as $p \to \infty$,  in view of Theorem \ref{behaviourofscattering}. 
\end{exam}
\vskip 0.2cm
The behaviour of scattering lengths for small potentials is of independent of interest. Let us consider the case that $\mu({\rm d}x)=V(x){\rm d}x$ with $V$ being a positive $L^1(\R^d)$-function. 

\begin{prop}\label{small}
Suppose that $V$ and ${\bf F}1$ have bounded supports in $\R^d$. Then  
\begin{align*}
\lim_{\varepsilon \to 0}\frac{1}{\varepsilon}{\sf \Gamma}(\varepsilon V +\varepsilon F) = \int_{\R^d}\left(V + \widehat{\bf F}1\right)(x){\rm d}x,
\end{align*}
where $\widehat{\bf F}1(x):=C_{d,\alpha}\int_{\R^d}F(x,y)|x-y|^{-d-\alpha}{\rm d}y$. 
\end{prop}

\begin{proof}
It is clear from (\ref{express}) that ${\sf \Gamma}(V+F) \le \int_{\R^d}(V(x){\rm d}x+{\bf F}1(x){\rm d}x)$. 
For convenience, assume that $V$ and ${\bf F}1$ have a common bounded support $B \subset \R^d$. By \eqref{UtildeRzeta}, we see that
\begin{align}\label{UtildeR}
U_{V+F}(x):= \widetilde{R}\left(\EE_{\cdot}\left[e^{-\int_0^{\infty}V(X_t){\rm d}t - \sum_{t>0}F(X_{t-},X_t)}\right](V+ {\bf F}1)\right)(x).
\end{align}
It is known in \cite[Corollary 2.8]{Song} that the resolvent kernel $\widetilde{r}(x,y)$ of $\widetilde{R}$ satisfies
\begin{align}\label{estRtilde}
\frac{C^{-1}}{|x-y|^{d-\alpha}} \le \widetilde{r}(x,y) \le \frac{C}{|x-y|^{d-\alpha}}, \quad x,y \in \R^d
\end{align} 
for some $C >0$. 
Then  
\begin{align}
\int_{B}U_{V+F}(x){\rm d}x&=\int_B \widetilde{R}\left(\EE_{\cdot}\left[e^{-\int_0^{\infty}V(X_t){\rm d}t - \sum_{t>0}F(X_{t-},X_t)}\right](V+ {\bf F}1)\right)\!(x){\rm d}x \nonumber \\
&\le C\int_B\int_{\R^d}\frac{1}{|x-y|^{d-\alpha}}\EE_{y}\left[e^{-\int_0^{\infty}V(X_t){\rm d}t - \sum_{t>0}F(X_{t-},X_t)}\right](V+{\bf F}1)(y){\rm d}y\,{\rm d}x \nonumber \\
&\le C(B){\sf \Gamma}(V+F) \le C(B)\int_{\R^d}(V+{\bf F}1)(x){\rm d}x,  
\label{UFdomi}
\end{align}
where $C(B):=C\sup_{y \in \R^d}\int_B |x-y|^{\alpha -d}{\rm d}x$. From this and the fact that ${\bf F}^{(\varepsilon)}1 \to 0$ as $\varepsilon \to 0$, we see  
\begin{align*}
\int_B U_{\varepsilon V+ \varepsilon F}(x){\rm d}x \longrightarrow 0, \quad {\rm as} \!\!\quad \varepsilon \to 0,
\end{align*}
which implies that $U_{\varepsilon V+ \varepsilon F} \to 0$ a.e. on $B$. The same are true for $U_{\varepsilon V+ \varepsilon F}V$ and $\varepsilon^{-1}U_{\varepsilon V+ \varepsilon F}{\bf F}^{(\varepsilon)}1$ because $\varepsilon^{-1}{\bf F}^{(\varepsilon)}1 \to \widehat{\bf F}1$ as $\varepsilon \to 0$. 
Now, by the definition of the scattering length and the dominated convergence theorem, 
\begin{align*}
&\left|\int_{\R^d}\left(V+\varepsilon^{-1}{\bf F}^{(\varepsilon)}1\right)(x){\rm d}x - \varepsilon^{-1}{\sf \Gamma}(\varepsilon V +\varepsilon F)\right|  \\
&\quad \quad \quad = \int_{B}U_{\varepsilon V+\varepsilon F}(x)\left(V+\varepsilon^{-1}{\bf F}^{(\varepsilon)}1\right)(x){\rm d}x \longrightarrow 0, \quad {\rm as}\!\!\quad \varepsilon \to 0.
\end{align*}
The proof is complete.
\end{proof}

\section{Bounds for the bottom of the spectrum via scattering length}\label{sec:EigenEst}  
In the rest of the sections, we are going to consider the case that $\mu$ is an absolutely continuous measure with respect to the Lebesgue measure having $V \ge 0$ as a density function, that is, $\mu({\rm d}x)=V(x){\rm d}x$. 

In this section, we give a two-sided bound for the bottom of the spectrum of the fractional Neumann Laplacian on the unit cube $D:=D_{1,0}$,  the cube of side length $1$ centered at $0$, in $\R^d$ with local and non-local perturbations. We assume that $V$ and ${\bf F}1$ are integrable functions supported on $D$.

Let $(-\Delta)^{\alpha /2}_N$ be the fractional Laplacian on $L^2(D)$ with the Neumann boundary condition, as the generator of the Dirichlet form on $L^2(D)$ defined by
\begin{align*}
&\E^{\rm ref}(f,g)=\frac{C_{d,\alpha}}{2}\int_{D}\int_{D}\frac{(f(x)-f(y))(g(x)-g(y))}{|x-y|^{d+\alpha}}{\rm d}x{\rm d}y \\
&\F^{\rm ref}_a=\left\{f \in L^2(D) : \int_{D}\int_{D}\frac{(f(x)-f(y))^2}{|x-y|^{d+\alpha}}{\rm d}x{\rm d}y <\infty\right\}.
\end{align*} 
It is known that $(\E^{\rm ref}, \F^{\rm ref}_a)$ is the active reflected Dirichlet form of $(\E^D, \F^D)$, the Dirichlet form on $L^2(D)$ associated with the part process ${\bf X}^D$ of ${\bf X}$ on $D$. The stochastic process ${\bf Y}=(Y_t, \PP_x)$ associated with $(\E^{\rm ref}, \F^{\rm ref}_a)$ (or $(-\Delta)^{\alpha /2}_N$) is then the reflected stable process on $D$ (cf. \cite{BBC}).

Define the bottom of the spectrum of the formal Schr${\rm \ddot{o}}$dinger operator ${\mathcal L}_{V+F}^N:=(-\Delta)^{\alpha /2}_N + V+ d{\bf F}$ on $L^2(D)$ by
\begin{align}\label{Neigen}
\lambda_1^N(V+F):=\inf_{\varphi \in \F^{\rm ref}_a}\frac{\E^{\rm ref}(\varphi, \varphi)+{\mathcal H}_D^{V+F}(\varphi, \varphi)}{\int_D\varphi(x)^2{\rm d}x}, 
\end{align} 
where
\begin{align}\label{pertubform}
{\mathcal H}_D^{V+F}(\varphi, \varphi)=\int_D\varphi(x)^2V(x){\rm d}x+\int_{D}\int_{D}\frac{\varphi(x)\varphi(y)\left(1-e^{-F(x,y)}\right)}{|x-y|^{d+\alpha}}{\rm d}x{\rm d}y.
\end{align}
\vskip 0.2cm
The next two propositions are originally due to \cite{Siudeja:Illinois} when $F\equiv 0$ (also \cite{Taylor:1976} for the Brownian motion). The proofs can be deduced by some modifications to the perturbation term. First, we give a upper bound for $\lambda_1^N(V+F)$.

\begin{prop}\label{upper}
There exists a constant $C_2(D) >0$ such that
\begin{align*}
\lambda_1^N(V+F) \le C_2(D){\sf \Gamma}(V+F)
\end{align*}
provided ${\sf \Gamma}(V+F)$ is small.
\end{prop}

\begin{proof}
Note that the infimum in \eqref{Neigen} may be taken over $L^2(D)$. Put $\varphi:=1-U_{V+F}$. Clearly $\varphi \in L^2(D)$. Let $V_n:=\1_{B(0,n)}(V\wedge n)$ and $F_{n}(\cdot, y):=\1_{B(0,n)}(F(\cdot,y)\wedge n)$ for $y \in \R^d$, where $B(0,n)$ is the open ball in $\R^d$ with center $0$ and radius $n$. Since $V_n$ and ${\bf F}_n1$ belong to $L^2(\R^d)$, the relation \eqref{DEforUF} is rigorously established for $V_n$ and $F_n$.  Therefore we have by Fatou's lemma that
\begin{align*}
&\lambda_1^N(V+F)\int_D\varphi(x)^2{\rm d}x \\
&\quad \!\le \frac{C_{d,\alpha}}{2}\int_{\R^d}\int_{\R^d}\liminf_{n \to \infty}\frac{(U_{V_n+F_n}(x)-U_{V_n+F_n}(y))^2}{|x-y|^{d+\alpha}}{\rm d}x{\rm d}y \\
&\quad \quad \!\!+ \int_{\R^d}\liminf_{n\to \infty}(1-U_{V_n+F_n})^2(x)V_n(x){\rm d}x+\int_{\R^d}\liminf_{n\to \infty}(1-U_{V_n+F_n})(x){\bf F}_n(1-U_{V_n+F_n})(x){\rm d}x \\
&\quad \!\le \frac{C_{d,\alpha}}{2}\liminf_{n\to \infty}\int_{\R^d}\int_{\R^d}\frac{(U_{V_n+F_n}(x)-U_{V_n+F_n}(y))^2}{|x-y|^{d+\alpha}}{\rm d}x{\rm d}y \\
&\quad \quad \!\!+ \liminf_{n\to \infty}\int_{\R^d}(1-U_{V_n+F_n})^2(x)V_n(x){\rm d}x+\liminf_{n\to \infty}\int_{\R^d}(1-U_{V_n+F_n})(x){\bf F}_n(1-U_{V_n+F_n})(x){\rm d}x \\
&\quad \!= \liminf_{n\to \infty}\left\{\int_{\R^d}U_{V_n+F_n}(x)(-\Delta)^{\alpha /2}U_{V_n+F_n}(x){\rm d}x + {\sf \Gamma}(V_n+F_n)\right. \\ 
&\quad \quad \!\!\left.-\int_{\R^d}U_{V_n+F_n}(x)(1-U_{V_n+F_n})(x)V_n(x){\rm d}x - \int_{\R^d}U_{V_n+F_n}(x){\bf F}_n(1-U_{V_n+F_n})(x){\rm d}x\right\}. \\
&\quad \!= \liminf_{n\to \infty}{\sf \Gamma}(V_n+F_n).
 \end{align*} 
The sequences $V_n$ and $F_n$ converge to $V$ and $F$, respectively. Thus by the monotonicities of the capacitary potential and the scattering length we see that $U_{V_n+F_n} \nearrow U_{V+F}$ and ${\sf \Gamma}(V_n+F_n) \nearrow {\sf \Gamma}(V+F)$ as $n \to \infty$, respectively. Hence we have 
\begin{align}\label{lambdatogamma}
\lambda_1^N(V+F)\int_D(1-U_{V+F})^2(x){\rm d}x \le \liminf_{n\to \infty}{\sf \Gamma}(V_n+F_n)={\sf \Gamma}(V+F).
\end{align}
From \eqref{UFdomi}, it holds that
\begin{align*}
\int_D(1-U_{V+F})^2(x){\rm d}x \ge 1 -2\int_DU_{V+F}(x){\rm d}x \ge 1 -2C(D){\sf \Gamma}(V+F), 
\end{align*}
Applying this to \eqref{lambdatogamma}, we obtain that 
\begin{align*}
\lambda_1^N(V+F)(1 -2C(D){\sf \Gamma}(V+F)) \le {\sf \Gamma}(V+F). 
\end{align*}
Now the assertion holds if we make ${\sf \Gamma}(V+F)$ so small that ${\sf \Gamma}(V+F) \le 1/(4C(D))$.  
\end{proof}

\begin{rem}\label{Notcube}
The result of Proposition \ref{upper} is valid for any bounded domain $D$.
\end{rem}
\vskip 0.3cm
Next, we turn to a lower bound for $\lambda_1^N(V+F)$. To do this, we need some facts on subordinated processes. Let ${\bf B}=(B_t, \PP_x)$ be a Brownian motion in $\R^d$ running twice the usual speed. Let ${\bf Z}=(Z_t, \PP_x)$ be a reflected Brownian motion on $D$, that is, ${\bf Z}$ is the process generated by the Laplacian with the Neumann boundary condition in $D$. We will derive symmetric stable processes from ${\bf B}$ and ${\bf Z}$ by using a subordination technique. Let $S_t$ be a positive $\alpha /2$-stable subordinator independent of ${\bf B}$ and ${\bf Z}$. Then the symmetric $\alpha$-stable process ${\bf X}$ is nothing but the subordinated process of ${\bf B}$ by $S_t$, that is, $X_t=B_{S_t}$. Let ${\bf W}=(W_t, \PP_x)$ be the subordinated process of ${\bf Z}$ by $S_t$. It is known that ${\bf W}$ is a stable-like process studied in \cite{ChenKumagai}, but it is, in general, different from ${\bf Y}$ the reflected stable process on $D$ associated with  $(\E^{\rm ref}, \F^{\rm ref}_a)$ (or $(-\Delta)^{\alpha /2}_N$) (cf. \cite{BBC}). 

Let denote by $-{\mathcal W}$ the generator of ${\bf W}$ and $\lambda_1^{\mathcal W}(V+F)$ the bottom of the spectrum of the Schr${\rm \ddot{o}}$dinger operator $-{\mathcal W} + V+d{\bf F}$. It is easy to check that for some $c_2> c_1 >0$
\begin{align}\label{lambdaVN}
c_1\lambda_1^{\mathcal W}(V+F) \le \lambda_1^N(V+F) \le c_2\lambda_1^{\mathcal W}(V+F).
\end{align}
Moreover, we can prove the following relation by using a simlilar method as in \cite[Lemma 4.1]{Siudeja:Illinois}: for any $t>0$
\begin{align}\label{Ptcomparison}
\EE_x\left[e^{-\int_0^tV(W_s){\rm d}s-\sum_{0<s \le t}F(W_{s-}, W_s)}\right] \le 1-U_{V+F}^T(x).
\end{align}
\vskip 0.3cm
\begin{prop}\label{lower}
There exists a constant $C_1(D)>0$ such that
\begin{align*}
C_1(D){\sf \Gamma}(V+F) \le \lambda_1^N(V+F).
\end{align*}
\end{prop}

\begin{proof}
In view of \eqref{lambdaVN}, it is enough to show that there exists a constant $c(D) >0$ such that $c(D){\sf \Gamma}(V+F) \le \lambda_1^{\mathcal W}(V+F)$. To do this, we prove that there exist $T>0$ and a constant $c_1(D)>0$ such that
\begin{align}\label{Target}
\sup_{x \in D}\left(1-U_{V+F}^T(x)\right) \le e^{-c_1(D){\sf \Gamma}(V+F)}.
\end{align}
Then, we see from \eqref{Ptcomparison} that
\begin{align*}
\left\|e^{-T(V+d{\bf F}-{\mathcal W})}\right\|_{2,2} &\le \sup_{x \in D}\EE_{x}\left[e^{-\int_0^TV(W_s){\rm d}s-\sum_{s \le T}F(W_{s-}, W_s)}\right] \\
&\le \sup_{x \in D}\left(1-U_{V+F}^T\right) \le e^{-c_1(D){\sf \Gamma}(V+F)},
\end{align*}
which implies $c(D){\sf \Gamma}(V+F) \le \lambda_1^{\mathcal W}(V+F)$. Here $\|\cdot \|_{2,2}$ means the operator norm from $L^2(D)$ to $L^2(D)$. 
Now we shall prove \eqref{Target}. By the Markov property, for any $t, s >0$ 
\begin{align*}
&U_{V+F}^{t+s}(x)-U_{V+F}^t(x) \\
&\quad =\EE_x\left[e^{-\int_0^tV(X_u){\rm d}u-\sum_{0<u \le t}F(X_{u-},X_u)}\right] -\EE_x\left[e^{-\int_0^{t+s}V(X_u){\rm d}u-\sum_{0<u \le t+s}F(X_{u-},X_u)}\right] \\
&\quad =\EE_x\left[e^{-\int_0^tV(X_u){\rm d}u-\sum_{0<u \le t}F(X_{u-},X_u)}\left(1-\EE_{X_t}\left[e^{-\int_0^sV(X_u){\rm d}u-\sum_{0<u \le s}F(X_{u-},X_u)}\right]\right)\right] \\
&\quad =\EE_x\left[e^{-\int_0^tV(X_s){\rm d}s-\sum_{0<s \le t}F(X_{s-},X_s)}U_{V+F}^s(X_t)\right] \le \EE_x\left[U_{V+F}^s(X_t)\right].
\end{align*}
Letting $s \to \infty$, we have by \eqref{UtildeR} and \eqref{estRtilde} that
\begin{align}\label{uniformdiffer}
&U_{V+F}(x)-U_{V+F}^t(x) \le \EE_x\left[U_{V+F}(X_t)\right] \nonumber \\
& \quad \quad =\EE_x\left[\widetilde{R}\left(\EE_{\cdot}\left[e^{-\int_0^{\infty}V(X_t){\rm d}t - \sum_{t>0}F(X_{t-},X_t)}\right](V+ {\bf F}1)\right)(X_t)\right] \nonumber \\ 
& \quad \quad =\int_{\R^d}p_t(x,y)\widetilde{R}\left(\EE_{\cdot}\left[e^{-\int_0^{\infty}V(X_t){\rm d}t - \sum_{t>0}F(X_{t-},X_t)}\right](V+ {\bf F}1)\right)(y){\rm d}y \nonumber \\
& \quad \quad \le C\int_{\R^d}p_t(x,y)\int_{\R^d}\frac{\EE_{z}[e^{-\int_0^{\infty}V(X_t){\rm d}t - \sum_{t>0}F(X_{t-},X_t)}](V+ {\bf F}1)(z)}{|y-z|^{d-\alpha}}{\rm d}z{\rm d}y \nonumber \\
& \quad \quad \le C\left(\sup_{x \in D, z \in \R^d}\int_{\R^d}\frac{p_t(x,y)}{|y-z|^{d-\alpha}}{\rm d}y\right){\sf \Gamma}(V+F).
\end{align}
In view of \eqref{ptbound}, since
\begin{align*}
\int_{\R^d}\frac{p_t(x,y)}{|y-z|^{d-\alpha}}{\rm d}y &\le C\int_{\R^d}\int_0^\infty p_t(x,y)p_{\tau}(y,z){\rm d}\tau {\rm d}y \\
&=C\int_t^\infty p_{\tau}(x,z){\rm d}\tau \le C'\int_t^\infty \tau^{-d/\alpha }{\rm d}\tau =C't^{-d/\alpha+1} \longrightarrow 0, \quad \text{as}\quad \!\!\!t \to \infty,
\end{align*}
the left hand side of \eqref{uniformdiffer} converges to 0 as $t \to \infty$, locally uniformly in $x$. Thus we can take large enough $T>0$ so that $U_{V+F}-U_{V+F}^T \le U_{V+F}/2$. On the other hand, we have by \eqref{UtildeR} and \eqref{estRtilde} again
\begin{align*}
U_{V+F}(x)&=\widetilde{R}\left(\EE_{\cdot}\left[e^{-\int_0^{\infty}V(X_t){\rm d}t - \sum_{t>0}F(X_{t-},X_t)}\right](V+ {\bf F}1)\right)(x) \\ 
&\ge C^{-1}\int_{D}\frac{\EE_{y}[e^{-\int_0^{\infty}V(X_t){\rm d}t - \sum_{t>0}F(X_{t-},X_t)}](V+ {\bf F}1)(y)}{|x-y|^{d-\alpha}}{\rm d}y \\
&\ge C^{-1}\left({\rm diam}\, D\right)^{\alpha -d}{\sf \Gamma}(V+F), \quad x \in D.
\end{align*}
Hence we obtain
\begin{align*}
1-U_{V+F}^T(x) \le 1 - \frac{1}{2}U_{V+F}(x) \le 1- c(D){\sf \Gamma}(V+F) \le e^{-c_1(D){\sf \Gamma}(V+F)}
\end{align*}
uniformly in $x \in D$. The proof is complete. 
\end{proof}

\section{Equivalent criteria for discrete spectrum via scattering length}\label{sec:DiscreteSpec}  
In this section, we give an equivalent characterization for discreteness of the spectrum of the formal Schr${\rm \ddot{o}}$dinger operator 
\begin{align*}
{\mathcal L}_{V+F}:=(-\Delta)^{\alpha /2}+V+d{\bf F}
\end{align*}
in terms of the scattering length, by using the results obtained in the previous section.

Let $D_{r, \xi}$ be the $d$-dimensional cube of the form 
\begin{align*}
D_{r,\xi}=\left\{x \in \R^d : \left|\xi_j -x_j\right| \le \frac{r}{2},~j=1,2, \cdots, d \right\}, \quad \xi \in \R^d.
\end{align*}
In the sequel, we use the notation $(-\Delta)^{\alpha /2}_{N, r, \xi}$ for the Neumann fractional Laplacian on $L^2(D_{r,\xi})$, to emphasize the dependence of its side length $r$ and center $\xi$. Similarly to \eqref{Neigen}, we write $\lambda_1^{N,r,\xi}(V+F)$ for the bottom of the spectrum of $(-\Delta)^{\alpha /2}_{N,r,\xi}+V+d{\bf F}$. Let $V_{r,\xi}$ be a function supported on $D_{0,1}$ given by $V_{r,\xi}(x)=V(rx+\xi)$ and $F_{r,\xi}$ a non-negative function supported on $D_{1,0}\times D_{1,0}$ given by $F_{r,\xi}(x,y):=F(rx+\xi, ry+\xi)$. Then, by the definition, we can easily check that
\begin{align*}
\lambda_1^{N,r,\xi}(V+F)=r^{-\alpha}\lambda_1^{N,1,0}(r^{\alpha}V_{r,\xi}+F_{r,\xi}).
\end{align*}

First, we give the sufficient condition for the discreteness of the spectrum of ${\mathcal L}_{V+F}$ in terms of the scattering length relative to $V$ and $F$ restricted to cubes. By $\sigma_{\rm ess}(\H)$ we mean the essential spectrum set of a operator ${\mathcal H}$. Let ${\bf F}_{r,\xi}$ be the non-local linear operator defined in \eqref{dfnonlocaloper} for $F_{r,\xi}$.

\begin{prop}\label{suff}
Let $C_1(D_{1,0})$ be the positive constant as in Proposition \ref{lower}. Suppose that for given $c>0$ there exists $r:=r(c) \in (0,1]$ and $R:=R(c) >0$ such that
\begin{align}\label{suffGammacond}
C_1(D_{1,0}){\sf \Gamma}(r^{\alpha}V_{r,\xi}+F_{r,\xi}) \ge r^\alpha c \qquad \text{for} \hspace{0.2cm}|\xi| \ge R.
\end{align}
Then $\sigma ({\mathcal L}_{V+F}) =\sigma_{\rm d} ({\mathcal L}_{V+F})$. 
\end{prop}

\begin{proof}
Let us denote by $\lambda_{\rm ess}(V+F)$ the bottom of the set $\sigma_{\rm ess} ({\mathcal L}_{V+F})$. To end the proof, we will show that $\lambda_{\rm ess}(V+F)=\infty$ under \eqref{suffGammacond}. It follows from  
\eqref{suffGammacond} and Proposition \ref{lower} that there exists a constant $C_1(D_{1,0}) >0$ such that 
\begin{align*}
r^{-\alpha}\lambda_1^{N,1,0}(r^{\alpha}V_{r,\xi}+F_{r,\xi}) \ge r^{-\alpha}C_1(D_{1,0}){\sf \Gamma}(r^{\alpha}V_{r,\xi}+F_{r,\xi}) \ge c \quad \text{for} \hspace{0.2cm}|\xi| \ge R
\end{align*} 
and which yields that
\begin{align}\label{sigmaGamma}
\sigma \left(r^{-\alpha}\left((-\Delta)^{\alpha/2}_{N,1,0}+r^{\alpha}V_{r,\xi}+d{\bf F}_{r,\xi}\right)\right) \subset  [c, \infty)\quad \text{for} \hspace{0.2cm}|\xi| \ge R.
\end{align}
Note that the operators $(-\Delta)^{\alpha/2}_{N,r,\xi}+V+d{\bf F}$ and $r^{-\alpha}\left((-\Delta)^{\alpha/2}_{N,1,0}+r^{\alpha}V_{r,\xi}+d{\bf F}_{r,\xi}\right)$ are unitarily equivalent. Therefore, \eqref{sigmaGamma} is equivalent to 
\begin{align}\label{unitary}
\sigma \left((-\Delta)^{\alpha/2}_{N, r, \xi}+V+d{\bf F}\right) \subset [c, \infty) \quad \text{for} \hspace{0.2cm}|\xi| \ge R.
\end{align}
By a standard argument involving Rellich's theorem, \eqref{unitary} implies that $\sigma_{\rm ess} ({\mathcal L}_{V+F}) \subset [c,\infty)$, that is, $\lambda_{\rm ess}(V+F) \ge c$. Since $c$ is arbitrary, we have the assertion by letting $c \to \infty$.
\end{proof}

Now we turn to the necessary condition. Let $C_0^\infty(\R^d)$ be the set of all $C^\infty$ functions with compact support on $\R^d$. Set
\begin{align*}
M:=\left\{f \in C_0^\infty (\R^d) : {\mathcal E}(f,f) + {\mathcal H}^{V+F}(f,f) \le 1\right\},
\end{align*}
where ${\mathcal H}^{V+F}(f,f)$ is the bilinear form defined as in \eqref{pertubform} with $D$ replaced by $\R^d$. In a similar way of \cite[Lemma 2.2]{KS}, we see that $\sigma ({\mathcal L}_{V+F}) =\sigma_{\rm d} ({\mathcal L}_{V+F})$ if and only if $M$ is precompact in $L^2(\R^d)$. 

For the Schr${\rm \ddot{o}}$dinger operator $(-\Delta)^{\alpha /2}_{D,r,\xi}+V+d{\bf F}$, where $(-\Delta)^{\alpha /2}_{D, r, \xi}$ is the  fractional Laplacian on $L^2(D_{r,\xi})$ with the Dirichlet boundary condition, we denote by $\lambda_1^{D,r,\xi}(V+F)$ the bottom of its spectrum. 

\begin{lem}\label{precpt}
If $\sigma ({\mathcal L}_{V+F}) =\sigma_{\rm d} ({\mathcal L}_{V+F})$, then for each $r \in (0,1]$, $\lambda_1^{D,r,\xi}(V+F) \to \infty$ as $|\xi| \to \infty$.
\end{lem}
\vskip 0.2cm
\begin{proof}~~The proof is similar to that of \cite[Lemma 2.3 and Proposition 2.4]{KS}. We address here the proof for reader's convenience. For a fixed $x_0 \in \R^d$ and $R>0$, the compact embedding theorem says that $M_B=\{f |_{B(x_0,R)} \in L^2(B(x_0,R)) : f \in M\}$ is precompact in $L^2(B(x_0,R))$. From this fact, we see that the precompactness of $M$ is equivalent to the precompactness of $M_B$ with the condition: for any $\varepsilon >0$ there exists $R:=R(\varepsilon) >0$ such that
\begin{align}\label{cptvanish}
\int_{B(x_0,R)^c}f(x)^2{\rm d}x \le \varepsilon, \quad \text{for any}\hspace{0.2cm} f \in M.
\end{align} 
Now, choose a sufficiently small $\varepsilon >0$ and assume that $D_{r,\xi} \subset B(x_0,R)^c$ where $R=R(\varepsilon)$ corresponds to $\varepsilon$ according to \eqref{cptvanish}. Then for any $f \in C_0^\infty (D_{r,\xi})$ with ${\mathcal E}_{D_{r,\xi}}(f,f) + {\mathcal H}^{V+F}_{D_{r,\xi}}(f,f) \le 1$ we have $\int_{D_{r,\xi}}f(x)^2{\rm d}x \le \varepsilon$. Therefore
\begin{align*}
\frac{{\mathcal E}_{D_{r,\xi}}(f,f) + {\mathcal H}^{V+F}_{D_{r,\xi}}(f,f)}{\int_{D_{r,\xi}}f(x)^2{\rm d}x} \ge \frac{1}{\varepsilon}.
\end{align*}
This implies that $\lambda_1^{D,r,\xi}(V+F) \to \infty$ as $|\xi| \to \infty$.
\end{proof}

\begin{lem}\label{asympN}
If for each $r \in (0,1]$ $\lambda_1^{D,r,\xi}(V+F) \to \infty$ as $|\xi| \to \infty$, then so does $\lambda_1^{N,r,\xi}(V+F)$.
\end{lem}

\begin{proof}
Let $\lambda:=\lambda_1^{D,1,0}(r^{\alpha}V_{r,\xi}+F_{r,\xi})$. For a fixed $a \in (0,1)$, set $\tau:=\lambda^{-a}$. Since the transition density $p_t(x,y)$ of $(-\Delta)^{\alpha /2}$ satisfies \eqref{ptbound} for $x,y \in D_{1,0}$ and $t>0$, we have for some $C>0$
\begin{align*}
\left\|e^{-\tau \left(r^{\alpha}V_{r,\xi}+d{\bf F}_{r,\xi}+(-\Delta)^{\alpha /2}_{D,1,0}\right)}\right\|_{1,2} \le C\tau^{-d/\alpha}, \quad  \left\|e^{-\tau \left(r^{\alpha}V_{r,\xi}+d{\bf F}_{r,\xi}+(-\Delta)^{\alpha /2}_{D,1,0}\right)}\right\|_{2,\infty} \le C\tau^{-d/\alpha},
\end{align*}
while the definition of $\lambda$ gives
\begin{align*}
\left\|e^{-\tau \left(r^{\alpha}V_{r,\xi}+d{\bf F}_{r,\xi}+(-\Delta)^{\alpha /2}_{D,1,0}\right)}\right\|_{2,2} \le e^{-\tau \lambda}.
\end{align*}
Here $\|\cdot \|_{p,q}$ means the operator norm from $L^p(D_{1,0})$ to $L^q(D_{1,0})$ for $1 \le p,q \le \infty$. Then
\begin{align*}
&\left\|e^{-3\tau \left(r^{\alpha}V_{r,\xi}+d{\bf F}_{r,\xi}+(-\Delta)^{\alpha /2}_{D,1,0}\right)}1\right\|_{\infty} \\
&\qquad =\left\|e^{-2\tau \left(r^{\alpha}V_{r,\xi}+d{\bf F}_{r,\xi}+(-\Delta)^{\alpha /2}_{D,1,0}\right)}e^{-\tau \left(r^{\alpha}V_{r,\xi}+d{\bf F}_{r,\xi}+(-\Delta)^{\alpha /2}_{D,1,0}\right)}1\right\|_{\infty} \\
&\qquad \le \left\|e^{-2\tau \left(r^{\alpha}V_{r,\xi}+d{\bf F}_{r,\xi}+(-\Delta)^{\alpha /2}_{D,1,0}\right)}\right\|_{2,\infty}\left\|e^{-\tau \left(r^{\alpha}V_{r,\xi}+d{\bf F}_{r,\xi}+(-\Delta)^{\alpha /2}_{D,1,0}\right)}\right\|_{2,2} \\
&\qquad \le C\tau^{-d/\alpha}e^{-\tau \lambda}=C\lambda^{ad/\alpha}e^{-\lambda^{1-a}},
\end{align*}
that is, the transition density $p_t^{D,\, r^{\alpha}V_{r,\xi}+F_{r,\xi}}(x,y)$ of $(-\Delta)^{\alpha/2}_{D,1,0}+r^{\alpha}V_{r,\xi}+d{\bf F}_{r,\xi}$ satisfies
\begin{align*}
0 \le p_{3\tau}^{D,\,r^{\alpha}V_{r,\xi}+F_{r,\xi}}(x,y) \le C\lambda^{ad/\alpha}e^{-\lambda^{1-a}}, \quad \text{for}\hspace{0.2cm}x,y \in D_{1,0}.
\end{align*}
By $p_t^{N,\,r^{\alpha}V_{r,\xi}+F_{r,\xi}}(x,y)$, we denote the transition density of $(-\Delta)^{\alpha/2}_{N,1,0}+r^{\alpha}V_{r,\xi}+d{\bf F}_{r,\xi}$. Put 
\begin{align*}
q_t^{r^{\alpha}V_{r,\xi}+F_{r,\xi}}(x,y):=p_t^{N,\,r^{\alpha}V_{r,\xi}+F_{r,\xi}}(x,y)-p_t^{D,\,r^{\alpha}V_{r,\xi}+F_{r,\xi}}(x,y), \quad \text{for} ~~t \in (0,3\tau], ~x,y \in D_{1,0}. 
\end{align*}
Then $q_t^{r^{\alpha}V_{r,\xi}+F_{r,\xi}}(x,y)$ satisfies the following fractional heat equation:
\begin{align}\label{fractionalHE}
\left\{
\begin{array}{ll}\displaystyle{
\left(\partial_t  -(-\Delta)^{\alpha/2} -r^{\alpha}V_{r,\xi}-d{\bf F}_{r,\xi}\right)q_t^{r^{\alpha}V_{r,\xi}+F_{r,\xi}}(\cdot, y)=0 \quad \text{on}\hspace{0.2cm}(0,\infty)\times D_{1,0} }\\
~~\displaystyle{q_0^{r^{\alpha}V_{r,\xi}+F_{r,\xi}}(x,y)=0, \quad q_t^{r^{\alpha}V_{r,\xi}+F_{r,\xi}}(x,y)=p_t^{N,\,r^{\alpha}V_{r,\xi}+F_{r,\xi}}(x,y)~~~{\rm on}~~x \in \partial D_{1,0}. }
\end{array}
\right.
\end{align} 
Therefore 
\begin{align}\label{estN10}
0 \le q_t^{r^{\alpha}V_{r,\xi}+F_{r,\xi}}(x,y) \le p_t^{N,1,0}(x,y).\quad x \in \partial D_{1,0}, 
\end{align}
where $p_t^{N,1,0}(x,y)$ is the transition density of $(-\Delta)^{\alpha/2}_{N,1,0}$. Set
\begin{align*}
D_{1,0}^{(\tau)}:=\left\{y \in D_{1,0} : {\rm dist} (y, \partial D_{1,0}) \ge \tau^{1/(d+\alpha +1)}\right\}. 
\end{align*}
Then we see that 
\begin{align}\label{estimateN10}
p_t^{N,1,0}(x,y) \le \frac{Ct}{|x-y|^{d+\alpha}} \le  C\tau^{1-\frac{d+\alpha}{d+\alpha +1}}, \quad x \in \partial D_{1,0}, y \in D_{1,0}^{(\tau)},~t \in (0,3\tau].
\end{align}
So applying the maximum principle for the fractional Laplacian to \eqref{fractionalHE}, together with \eqref{estN10} and \eqref   {estimateN10}, gives
\begin{align*}
q_t^{r^{\alpha}V_{r,\xi}+F_{r,\xi}}(x,y) \le  C\tau^{1-\frac{d+\alpha}{d+\alpha +1}}, \quad x \in D_{1,0}, y \in D_{1,0}^{(\tau)},~t \in (0,3\tau],
\end{align*}
and hence, for sufficiently large $\lambda$, we have
\begin{align}\label{3tau}
0 \le p_{3\tau}^{N,\,r^{\alpha}V_{r,\xi}+F_{r,\xi}}(x,y)&=q_{3\tau}^{r^{\alpha}V_{r,\xi}+F_{r,\xi}}(x,y) + p_{3\tau}^{D,\,r^{\alpha}V_{r,\xi}+F_{r,\xi}}(x,y) \nonumber \\ 
&\le C\tau^{1-\frac{d+\alpha}{d+\alpha +1}} + C\lambda^{ad/\alpha}e^{-\lambda^{1-a}} \nonumber \\
&=\lambda^{-a\left(1-\frac{d+\alpha}{d+\alpha +1}\right)} + C\lambda^{ad/\alpha}e^{-\lambda^{1-a}} \nonumber \\
&\le C'\lambda^{ad/\alpha}e^{-\lambda^{1-a}},~~ x \in D_{1,0}, y \in D_{1,0}^{(\tau)}.
\end{align}
By using the semigroup property of $e^{-t(r^{\alpha}V_{r,\xi}+d{\bf F}_{r,\xi}+(-\Delta)^{\alpha/2}_{N,1,0})}$, the estimate \eqref{3tau} can be extended for any $t \in [3\tau, \infty)$, 
\begin{align}\label{beyond3tau}
0 \le p_t^{N,\,r^{\alpha}V_{r,\xi}+F_{r,\xi}}(x,y) \le C'\lambda^{ad/\alpha}e^{-\lambda^{1-a}}, \quad  x \in D_{1,0}, y \in D_{1,0}^{(\tau)}, ~t \in [3\tau, \infty).
\end{align}
In particular, if $\lambda$ is large enough that $3\tau = 3\lambda^{-a} <1$, then the estimate \eqref{beyond3tau} holds for $p_1^{N,\,r^{\alpha}V_{r,\xi}+F_{r,\xi}}(x,y)$. On the other hand, since
\begin{align*}
p_1^{N,\,r^{\alpha}V_{r,\xi}+F_{r,\xi}}(x,y) \le p_1^{N,1,0}(x,y) \le C, \quad x \in D_{1,0}, ~y \in D_{1,0} \setminus D_{1,0}^{(\tau)},
\end{align*}
it follows that
\begin{align*}
&\int_{D_{1,0}}p_1^{N,\,r^{\alpha}V_{r,\xi}+F_{r,\xi}}(x,y){\rm d}y \\
& \qquad = \int_{D_{1,0}^{(\tau)}}p_1^{N,\,r^{\alpha}V_{r,\xi}+F_{r,\xi}}(x,y){\rm d}y + \int_{D_{1,0}\setminus D_{1,0}^{(\tau)}}p_1^{N,\,r^{\alpha}V_{r,\xi}+F_{r,\xi}}(x,y){\rm d}y \\
&\qquad \le C\lambda^{ad/\alpha}e^{-\lambda^{1-a}} + C\lambda^{-a/(d+\alpha +1)}.
\end{align*}
We also have a similar bound for $\int_{D_{1,0}}p_1^{N,\,r^{\alpha}V_{r,\xi}+F_{r,\xi}}(x,y){\rm d}x$ by the symmetry. Hence we can deduce that
\begin{align*}
\left\|e^{-\left(r^{\alpha}V_{r,\xi}+d{\bf F}_{r,\xi}+(-\Delta)^{\alpha/2}_{N,1,0}\right)}\right\|_{2,2} 
\le C\lambda^{ad/\alpha}e^{-\lambda^{1-a}} + C\lambda^{-a/(d+\alpha +1)}
\end{align*}
which implies the assertion.
\end{proof}

We are now ready to prove the necessary condition for the discreteness of the spectrum of ${\mathcal L}_{V+F}$ in terms of the scattering length.

\begin{prop}\label{nece}
Suppose that $\sigma ({\mathcal L}_{V+F}) =\sigma_{\rm d} ({\mathcal L}_{V+F})$. 
Then for given $c>0$, there exists $r_0:=r_0(c) \in (0,1]$ and $R : (0,r_0] \to (0,\infty)$ such that
\begin{align}\label{Gammacond}
{\sf \Gamma}(r^{\alpha}V_{r,\xi}+F_{r,\xi}) \ge r^{\alpha}c \qquad \text{for}\hspace{0.2cm}|\xi| \ge R(r),~ r \in (0,r_0].
\end{align}
\end{prop}
\vskip 0.3cm
\begin{proof}~~We argue by contradiction. For given $c >0$, choose $r_0:=r_0(c)$ so small that Proposition \ref{upper} can be applied, so that 
${\sf \Gamma}(r^{\alpha}V_{r,\xi}+F_{r,\xi}) \le r_0^{\alpha}c$ implies 
\begin{align*}
\lambda_1^{N,1,0}(r^{\alpha}V_{r,\xi}+F_{r,\xi}) \le C(D_{1,0}){\sf \Gamma}(r^{\alpha}V_{r,\xi}+F_{r,\xi}).
\end{align*}
As a consequence, if ${\sf \Gamma}(r^{\alpha}V_{r,\xi}+F_{r,\xi}) \le r^{\alpha}c$ for $r \in (0, r_0]$, then
\begin{align}\label{contratarget}
\lambda_1^{N,r,\xi}(V+F)&=r^{-\alpha}\lambda_1^{N,1,0}(r^{\alpha}V_{r,\xi}+F_{r,\xi}) \nonumber \\
&\le r^{-\alpha}C(D_{1,0}){\sf \Gamma}(r^{\alpha}V_{r,\xi}+F_{r,\xi}) \le cC(D_{1,0}).
\end{align}
However, we cannot have the bound \eqref{contratarget} for large enough $|\xi|$ in view of Lemma \ref{precpt} and Lemma \ref{asympN}. The proof is complete.
\end{proof}
Now, we can finish the proof of Theorem \ref{maindiscrete}.
\vskip 0.2cm
\emph{Proof of Theorem \ref{maindiscrete}} :~
It is clear that \eqref{Gammacond} implies \eqref{suffGammacond}. Hence we can conclude from Proposition \ref{suff} and Proposition \ref{nece} that the assertions in Theorem \ref{maindiscrete} are equivalent. \qed

\begin{rem}\label{FKcompactremark}
It is well known that the compactness of a bounded semigroup of linear operators is equivalent to the discreteness of the spectrum of the corresponding generator. Therefore we see that the assertions \eqref{d1} and \eqref{d2} in Theorem \ref{maindiscrete} are equivalent to the compactness of the following non-local Feynman-Kac semigroup $p_t^{V+F}$ on $L^2(\R^d)$ 
\begin{align*}
p_t^{V+F}f(x)=\EE_x\left[e^{-\int_0^tV(X_s){\rm d}s-\sum_{0<s\le t}F(X_{s-},X_s)}f(X_t)\right].
\end{align*}
\end{rem}

We further discuss a sufficient condition for the discreteness of the spectrum of ${\mathcal L}_{V+F}$. 

\begin{cor}\label{suff;thin} Let $V \in S_{D_0}({\bf X})$ (resp., or $F \in J_{D_0}({\bf X})$). Assume that for any $c>0$ there exists $r=r(c) \in (0,1]$ and 
\begin{align}\label{thinatinfinity}
\lim_{|\xi| \to \infty}|\{V \le c\} \cap D_{r,\xi}|=0 \quad \left(resp.,~or~~\lim_{|\xi| \to \infty}|\{{\bf F}1 \le c\} \cap D_{r,\xi}|=0\right).
\end{align}
Then $\sigma({\mathcal L}_{V+F})=\sigma_{\rm d}({\mathcal L}_{V+F})$.
\end{cor}

\begin{proof}
First we prove the assertion for $V$. In view of the assumption \eqref{thinatinfinity}, it follows that for any $c>0$ there exists $r=r(c) \in (0,1]$ and $R=R(c) >0$ such that for $|\xi| \ge R$ 
\begin{align}\label{comthin}
|\{V > c\} \cap D_{r,\xi}| \ge \frac{1}{2}|D_{r,\xi}|=\frac{1}{2}|D_{r,0}|.
\end{align}
Hence, we have by \eqref{comthin} 
\begin{align*}
{\sf \Gamma}(r^{\alpha}V_{r,\xi}+F_{r,\xi}) &\ge {\sf \Gamma}(r^{\alpha}V_{r,\xi}) \\
&=r^{\alpha}\int_{D_{1,0}}\EE_x\left[e^{-r^{\alpha}\int_0^\infty V_{r,\xi}(X_s){\rm d}s}\right]V_{r,\xi}(x){\rm d}x \\
&=r^{\alpha -1}\int_{D_{r,\xi}}\EE_{r^{-1}(x-\xi)}\left[e^{-r^{\alpha}\int_0^\infty V_{r,\xi}(X_s){\rm d}s}\right]V(x){\rm d}x \\
&\ge r^{\alpha -1}\int_{D_{r,\xi}}e^{-r^{\alpha}\sup_{y \in D_{1,0}}\EE_y\left[\int_0^\infty V_{r,\xi}(X_s){\rm d}s\right]}V(x){\rm d}x \\
&\ge r^{\alpha -1}\int_{\{V > c\}\cap D_{r,\xi}}e^{-r^{\alpha}\sup_{y \in D_{r,\xi}}\EE_y\left[\int_0^\infty V(X_s){\rm d}s\right]}V(x){\rm d}x \\
&\ge r^{\alpha -1}e^{-r^{\alpha}\ell} \int_{\{V > c\}\cap D_{r,\xi}}V(x){\rm d}x \ge r^{\alpha -1}e^{-r^{\alpha}\ell}\frac{c}{2}|D_{r,0}|  
\end{align*}
which implies Theorem \ref{maindiscrete}\eqref{d1}. The proof of the assertion for $F$ can be deduced in a similar way, by using the fact that ${\bf F}_{r,\xi}1(r^{-1}(x-\xi))=r^{d+\alpha -1}{\bf F}1(x)$ on $D_{r,\xi}$. 
\end{proof}

The condition \eqref{thinatinfinity} means that the sublevel set $\{V\le c\}$ (resp. $\{{\bf F}1 \le c\}$) is to be thin at infinity. More results for compactness of Schr\"odinger operators based on the concept of thin at infinity can be found in \cite{LSW, TTT} (see also \cite{Simon09, WangWu} as special cases).

\vskip 0.5cm
\noindent
{\bf Acknowledgement.}
The authors would like to thank the referee for valuable comments and suggestions. The authors also thank to Professor Kazuhiro Kuwae for his helpful comments.  

\vspace{\baselineskip}

\vskip 1.0cm
\noindent
Daehong Kim \\
Department of Mathematics and Engineering, \\
Graduate School of Science and Technology,
Kumamoto University, \\
Kumamoto, 860-8555 Japan \\
{\tt daehong@gpo.kumamoto-u.ac.jp}
\vskip 0.6cm
\noindent
Masakuni Matsuura \\
National Institute of Technology, Kagoshima College, \\
Kirishima, 899-5193 Japan \\
{\tt matsuura@kagoshima-ct.ac.jp}

\end{document}